\documentclass[11pt]{article}
\usepackage{float}
\usepackage{color}
\usepackage{amsmath,amssymb,amsthm,mathrsfs,amsfonts}
\usepackage{graphicx}
\usepackage{color}

\def\'#1{\ifx#1i{\accent"13 \i}\else{\accent"13 #1}\fi}

\usepackage{url}

\def\P{{\rm I\!P}}
\def\L{{\rm I\!L}}
\def\l{{\rm I\!l}}

\newtheorem{theorem}{Theorem}
\newtheorem{lemma}{Lemma}
\theoremstyle{definition}

\title{Achromatic arboricity on complete graphs
\thanks{Research supported by PAPIIT-M\'exico IN108121 and CONACyT-M\'exico 282280.}}

\author{Gabriela Araujo-Pardo \footnotemark[3] \\ \url{garaujo@matem.unam.mx}
\and Christian Rubio-Montiel \footnotemark[2] \\ \url{christian.rubio@acatlan.unam.mx}}

\begin{document}
\maketitle

\def\thefootnote{\fnsymbol{footnote}}
\footnotetext[3]{Instituto de Matem{\' a}ticas, Universidad Nacional Aut{\' o}noma de M{\' e}xico, 04510 M{\' e}xico City, Mexico.}
\footnotetext[2]{Divisi{\' o}n de Matem{\' a}ticas e Ingenier{\' i}a, FES Acatl{\' a}n, Universidad Nacional Aut{\'o}noma de M{\' e}xico, 53150 Naucalpan, Mexico.}

\begin{abstract} 
In this paper we study the {\it {achromatic arboricity}} of the complete graph. This parameter arises from the arboricity of a graph as the achromatic index arises from the chromatic index. The achromatic arboricity of a graph $G$, denoted by $A_{\alpha}(G)$,  is the maximum number of colors that can be used to color the edges of $G$ such that every color class induces a forest but any two color classes contain a cycle. In particular, if $G$ is a complete graph we prove that \[\frac{1}{4}n^{\frac{3}{2}}-\Theta(n) \leq A_{\alpha}(G)\leq \frac{1}{\sqrt{2}}n^{\frac{3}{2}}-\Theta(n).\]
\end{abstract}


\section{Introduction}

Let $G$ be a finite simple graph. A $k$-\emph{coloring} of $G$ is a surjective function $\varsigma$ that assigns a number from the set $\{1,2,\dots,k\}$ to each vertex of $G$ such that any two adjacent vertices have different colors. A $k$-coloring $\varsigma$ is called \emph{complete} if for each pair of different colors $i,j\in \{1,2,\dots,k\}$ there exists an edge $xy\in E(G)$ such that $\varsigma(x)=i$ and $\varsigma(y)=j$. 

While the \emph{chromatic number $\chi(G)$ of $G$} is defined as the smallest number $k$ for which there exists a $k$-coloring of $G$, the \emph{achromatic number $\alpha(G)$ of $G$} is defined as the largest number $k$ for which there exists a complete $k$-coloring of $G$ (see \cite{MR272662}). Note that any $\chi(G)$-coloring of $G$ is also complete. Therefore, for any graph $G$ 
\[\chi(G)\leq \alpha(G).\]

In \cite{FHKS19} the authors introduce the parameter called $a$-vertex arboricity of a graph $G$, denoted as $ava(G)$, defined as the largest number of colors that can be used to color the vertices of $G$ such that every color induces a forest but merging any two yields a monochromatic cycle, this parameter arises from the vertex arboricity parameter, denoted by $va(G)$, which is defined as the minimal number of induced forests which cover all the vertices (see \cite{MR236049}), and clearly as a minimal decomposition of trees is complete, we have that 
\[va(G)\leq ava(G).\]

Inspired in these parameters and in our previous work related to edge complete colorings, most specifically with the achromatic (proper colorings), the pseudoachromatic (non proper colorings) and pseudoconnected  (connected and no proper colorings) indices of the complete graphs  \cite{MR3249588,MR3774452,MR2778722,MR3695270}; we define the \emph{achromatic arboricity} parameter for a graph $G$, denoted by $A_{\alpha}(G)$,  as the largest number of colors that can be used to color the edges of $G$ such that each color class is acyclic and any pair of color classes induces a subgraph with at least a cycle. 

Clearly, this parameter arises from the well-known arboricity parameter of a graph $G$, defined by Nash-William in 1961 \cite{MR133253,MR161333}, and denoted for $A(G)$, that is the minimum number of acyclic subgraphs into which $E(G)$ can be partitioned. Note that the union of two parts induced a subgraph with at least a cycle. 
In consequence, we have that 
\[A(G)\leq A_{\alpha(G)}.\]

In this paper we give a lower and an upper bound for the achromatic arboricity parameter with a small gap between them, more precisely,  we prove that $A_{\alpha}(K_n)\approx \frac{1}{c}n^{\frac{3}{2}}$ for $\sqrt{2}\leq c \leq 4$. 
 
This paper is organized as follows.  On Section \ref{s2} we give a a general upper bound. On Section \ref{s3} we give a lower bound using the properties and the structure of the finite projective planes.  On Section \ref{s4}, we prove our main theorem as a consequence of the previous results.  Finally,  on Section \ref{smallvalues} we give the exact values of the achromatic arboricity for $K_n$ when $2\leq n \leq 7$.

\section{The upper bound for the achromatic arboricity of $K_n$}\label{s2}

In this section we prove an upper bound for $A_{\alpha}(K_{n})$. The technique has been used previously by different authors in different papers cited in the introduction of this paper. 

\begin{lemma}\label{lema1}
If $n\geq 5$ then \[A_{\alpha}(K_{n})\leq\left\lfloor \max\left\{ \min\{f_n(x),g_n(x)\} \colon x\in\mathbb{N}\right\}\right\rfloor\] where $f_n(x)=n(n-1)/2x$ and $g_n(x)=x(n-x-1)+1$.
\end{lemma}
\begin{proof}
Let $\varsigma\colon E(K_n) \rightarrow [k]$ be an acyclic $k$-edge-coloring of $K_n$ such that the union of any two color classes induce at least a cycle, with $k=A_{\alpha}(K_n)$.  Let $x= min\{ \left|\varsigma^{-1}(i)\right| : i\in \left[k\right]\}$, i.e., $x$ is the cardinality of the smallest color class of $\varsigma$. Without loss of generality, let $x= \left|\varsigma^{-1}(k)\right|$ be. Since $\varsigma$ defines a partition in the edges of $K_n$ it follows that $k\leq f_n(x):=n(n-1)/2x$. 

Then, since $\varsigma$ is acyclic, we can suppose that $\varsigma^{-1}(k) $ induces a matching. Moreover, the number of pair of edges in the subgraph induced by $E(K_{2x})\setminus \varsigma^{-1}(k)$ is at most $\frac{\binom{2x}{2}-x}{2}=x(x-1)$. On the other hand, there are $2x(n-2x)$ edges incident to some vertex of $\varsigma^{-1}(k)$ exactly once, we denote this set of edges by	 $X$. Since every two color classes of $\varsigma$ have at least two incidences, there are at least two edges that have a vertex in common with some edge in $\varsigma^{-1}(k)$, hence, the number of color classes incident to $\varsigma^{-1}(k)$ containing some edge in $X$ is at most $x(n-2x)$ and the number of color classes incident to $\varsigma^{-1}(k)$ containing no edge in $X$ is at most $x(x-1)$ since two edges are required to obtain a cycle. Hence, there are at most $g_n(x)-1$ color classes incident with some edge in $\varsigma^{-1}(k)$ where \[g_n(x)-1:=x(n-2x)+x(x-1),\]
therefore \[g_n(x)-1=xn-x^2-x,\]
it follows that $k\leq g_n(x)=xn-x^2-x+1$. In consequence, we have  \[A_{\alpha}(K_{n})\leq\min\{f_n(x),g_n(x)\}.\]
And we conclude that
\[A_{\alpha}(K_{n})\leq\left\lfloor \max\left\{ \min\{f_n(x),g_n(x)\} \colon x\in\mathbb{N}\right\}\right\rfloor.\]
\end{proof}

The function $f_n$ is a hyperbola and the function $g_n$ is a parabola, see Figure \ref{Fig1}. Then we have the following lemma.

\begin{figure}[!htbp]
\begin{center}
\includegraphics{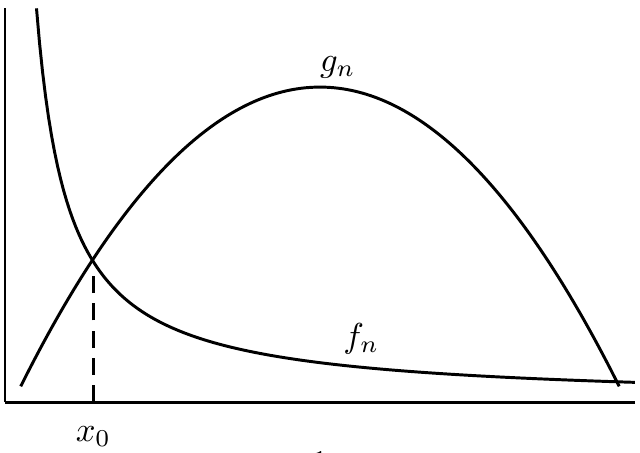}
\caption{The functions $g_n$ and $f_n$ for a fixed value $n$.}\label{Fig1}
\end{center}
\end{figure}

\begin{lemma}\label{lema2}
Let $x_0,x_1\in \mathbb{R}^{+}$, such that $f_n(x_0)=g_n(x_0)$, and $f_n(x_1)=g_n(x_1)$. If $x_0<x_1$ then \[g_n(x_0)=f_n(x_0)=\max\left\{ \min\{f_n(x),g_n(x)\} \colon x\in\mathbb{R}^{+}\right\}\]
where $x_0=\sqrt{(n+5/8)/2+\epsilon}+1/4$ and $\epsilon>0$.
\end{lemma}
\begin{proof}
Note that $f_n(x) \leq g_n(x)$ for $x_0 \leq x \leq x_1$ and $g_n(x)<f_n(x)$ in any other case for $x \in \mathbb{R}^{+}$. Since $f_n(x_0)>f_n(x_1)$ it follows that $g_n(x_0)=f_n(x_0)=\max\left\{ \min\{f_n(x),g_n(x)\} \colon x\in\mathbb{R}^{+}\right\}.$

The equation $f_n(x_0)=g_n(x_0)$ is reduced to $n^2-(2x_0^2+1)n+2x_0(x^2_0+x_0-1)=0$. 
Since the discriminant $D=4x_0^4-8x_0^3-4x_0^2+8x_0+1$, $\sqrt{D}=2x_0^2-2x_0-2-\epsilon$, for some $\epsilon>0$. Then the positive solution for $n$ is $n=2x_0^2-x_0-\frac{1}{2}-\frac{\epsilon}{2}$ and the lemma holds true because $x_0=\sqrt{(n+5/8)/2+\frac{\epsilon}{4}}+1/4$ provides the positive solution. 
\end{proof}

Now, we can prove the upper bound in the following theorem.

\begin{theorem}\label{teo1}
Let $n\geq 5$ be an integer then the achromatic arboricity of the complete graph of order $n$ is upper bounding by:
\[A_{\alpha}(K_{n})\leq \frac{1}{\sqrt{2}}n^{\frac{3}{2}}-\Theta(n).\]
\end{theorem}
\begin{proof}
By Lemma \ref{lema2},  $g_n(x_0)=nx_0-x_0^2-x_0+1$ and $x_0=\sqrt{n/2}+\epsilon$, for a small $\epsilon>0$. We obtain: 
\[g_n(x_0)=n(\sqrt{n/2}+\epsilon)-(\sqrt{n/2}+\epsilon)^2-(\sqrt{n/2}+\epsilon)+1,\]
and then $g_n(x_0)= \frac{1}{\sqrt{2}}n^{\frac{3}{2}}- \frac{n}{2}-(2\epsilon+1)\sqrt{\frac{n}{2}}+n\epsilon-\epsilon^2-\epsilon+1$ and the result follows.
\end{proof}

\section{A lower bound for the achromatic arboricity of $K_n$  for some values of $n$}\label{s3}
In this section we provide a lower bound for the achromatic arboricity of $K_n$, for some values $n$ greater than $13$. We use the well-known technique of identifying the structure of the finite projective plane with the complete graph in order to use its properties. 

First that all, we recall some definitions and properties of the projective planes that we will use along the proof for the lower bound. 

A \emph{projective plane} is a set of $n$ points and a set of $n$ lines, with the following properties.
\begin{enumerate} \item For any two distinct points there is exactly one line incident to both.
\item For any two distinct lines there is exactly one point incident to both.
\item There exist four points such that no line is incident to three of them or more.
\end{enumerate}

A projective plane has $n=q^2+q+1$ points, for a suitable $q$ number, and $n$ lines. Each line has $q+1$ points and each point belongs to $q+1$ lines; we say that $q$ is the \emph{order} of such plane, and a projective plane of order $q$ is denoted by $\Pi_q$. 

Let $\P$ be the set of points of $\Pi_q$ and let $\L= \{\l_1,\dots, \l_n\}$ the set of lines of $\Pi_q$. We identify the points of $\Pi_q$ with the set of vertices of the complete graph $K_n$.  Then, the set of points of each line of $\Pi_q$ induces a subgraph $K_{q+1}$ in $K_n$. Given a line $\l_i\in \L$, let $l_i=(V(l_i),E(l_i))$ be the subgraph of $K_n$ induced by the set of $q+1$ points of $\l_i$.  By the properties of the projective plane, for each pair $i,j\in [n]$, $|V(l_i)\cap V(l_j)| =1$ and $\{E(l_1), \dots, E(l_n)\}$ is a partition of the edges of $K_n$. Therefore, when we say that a graph $G$ isomorphic to $K_n$ is a \emph{representation of the projective plane} $\Pi_q$ means that $V(G)$ is identified with the  points of $\Pi_q$ and that there is a set of subgraphs (lines) $\{l_1, \dots, l_n\}$ of $G$, for which a line $\l_i$ of $\Pi_q$, $l_i$ is the induced subgraph by the set of points  of $\l_i$. 

Let us recall that any complete graph of even order $q+1$ admits a factorization by $\frac{q+1}{2}$ hamiltonian paths, see \cite{MR0411988}. This factorization of the complete graph $K_{q+1}$ can be used as an edge-coloring for the lines of $K_n$.

In particular, if $q$ is a prime power there exists a $\Pi_q$, that arises from finite fields $\mathbb{Z}_q$ for $q$ prime, and from $GF(q)$, the Galois Field of order $q$, when $q$ is a prime power. It is called the \emph{algebraic projective plane}, and  it is denoted by $PG(2,q)$ (see \cite{MR554919}).  Since the proof of Theorem \ref{teo4} only requires projective planes of order prime, we will use the algebraic projective plane $PG(2,q)$ for $q$ a prime number. 

Now we give a useful description $PG(2,q)$: Let $P$ and $L$ be two incident point and line that we will call the infinity point and line, respectively. Let $\{P_0,P_1,\ldots,P_{q-1}\}$ be the set of points such that each point is different to $P$ and incident to $L$. And let $\{L_0,L_1,\ldots,L_{q-1}\}$ be the set of lines such that each line is different to $L$ and incident to $P$. 

Moreover, let $\{(i,0),(i,1),\ldots,(i,q-1)\}$ be the set of points, different to $P$, incident to $L_i$; and let $\{[i,0],[i,1],\ldots,[i,q-1]\}$ be the set of lines, different to $L$, incident to $P_i$. 

The remaining lines are denoted as follows. The line $[a,b]$ is adjacent to all the points $(x,y)$ that satisfy $y=ax+b$ using the arithmetic of $\mathbb{Z}_q$. 

\begin{theorem}\label{teo2}
If $q$ is an odd prime number and $n=q^2+q+1$ then \[ \frac{q+1}{4}(n+1)\leq A(K_n).\]
\end{theorem}
\begin{proof}
Let $q$ be an odd prime number and $G$ (isomorphic to $K_n$) be a representation of the algebraic projective plane $PG(2,q)$.

First, we proceed giving a partition of the lines of the projective plane taking a single line and $\frac{q^2+q}{2}$ pairs of lines having an intersection point between them which is different for each pair, that is a triplet $(p,m,l)$ such that $p=m\cap l$.  And second, we give an acyclic edge-coloring of $G$ that uses this partition of the lines and attains the given lower bound concluding the proof as follows. 

\begin{itemize}
\item [(1)] The single line is $L_0$, and the first triplet is  $(P_0,[0,0],L)$.
\item [(2)] Take the set $X=\{P_1,\ldots ,P_{\frac{q-1}{2}}\}$, a subset of $\frac{q-1}{2}$ points of $L$, and call $A$ the set of lines between the point $(0,0)$ and $X$, that is $A=\{[i,0]: i\in \{1,\ldots, \frac{q-1}{2}\}\}$, and let $B$ a set of lines between the points of $X$ and the point $(0,2i-1)$, that is $B=\{[i,2i-1]:  i\in \{1,\ldots, \frac{q-1}{2}\}\}.$  
Then, we have a set $T_1$ of $\frac{q-1}{2}$ triplets of two lines and one intersection point between them: 
\[T_1=\{(P_i,[i,0],[i,2i-1])\colon  i\in \{1,\ldots, \frac{q-1}{2}\}\}.\]
\item[(3)] Take the intersection points between the lines $[i,2i-1]$ and $L_i$, such that:
\[(i,i^2+2i-1)=[i,2i-1]\cap L_i, \ \text{for}  \  i\in \{1,\ldots, \frac{q-1}{2}\}.\]
And the set of lines between these points $(i,i^2+2i-1)$ and $(0,2i)$, they are the lines $C=\{[\frac{i^2-1}{i},2i]:  i\in \{1,\ldots, \frac{q-1}{2}\}\}.$ Then, we have another set of of $\frac{q-1}{2}$ triplets of two lines and one intersection point between them: 
\[T_2=\{(i,i^2+2i-1),L_i,[\frac{i^2-1}{i},2i]\colon  i\in \{1,\ldots, \frac{q-1}{2}\}\}.\]
Note that we have covered all the lines $[i,2i-1]$ (containing the points $(0,2i-1)$ and $(i,i^2+2i-1)$) and all the lines $[\frac{i^2-1}{i},2i]$ (containing the points $(0,2i)$ and $(i,i^2+2i-1)$) for $1\leq i \leq \frac{q-1}{2}$. 
\item[(4)] For any vertex $(i,i^2+2i-t)\in L_i$ with $t\in \mathbb{Z}_q-\{1\}$, we have the lines $\{[\frac{i^2-t+1}{i},2i-1]\}$ between $(i,i^2+2i-t)$ and $(0,2i-1)$,  and the lines $\{[\frac{i^2-t}{i},2i]\}$ between $(i,i^2+2i-t)$ and $(0,2i)$ with $t\in \mathbb{Z}_q-\{1\}$. 

Now we define $(q-1)\frac{q-1}{2}$ triplets as follows: 
\[T_3=\{((i,i^2+2i-t),[\frac{i^2-t+1}{i},2i-1],[\frac{i^2-t}{i},2i])\colon \ t\ \in \mathbb{Z}_q-\{1\}, \text{and}\  i\in \{1,\ldots, \frac{q-1}{2}\}\}.\]

\item[(5)] The last set of triplets is defined by: 
\[T_4=\{((i,i^2),[i,0],L_i)\colon i\in \{\frac{q+1}{2},  \dots, q-1\}\}.\]

\item[(6)] Hence, we have: $1+|T_1|+|T_2|+|T_3|+|T_4|=1+\frac{q-1}{2}+\frac{q-1}{2}+(q-1)\frac{q-1}{2}+\frac{q-1}{2}$ triplets that cover a set of $2+2(q-1)+(q-1)^2+q-1=q^2+q$ lines. Then, all the lines of $PG(2,q)$ are covered except for the line $L_0$. 

The left side of Figure \ref{Fig2} shows a decomposition of the lines of $PG(2,3)$ while the right side shows $K_{13}$.   
\end{itemize}

\begin{figure}[!htbp]
\begin{center}
\includegraphics{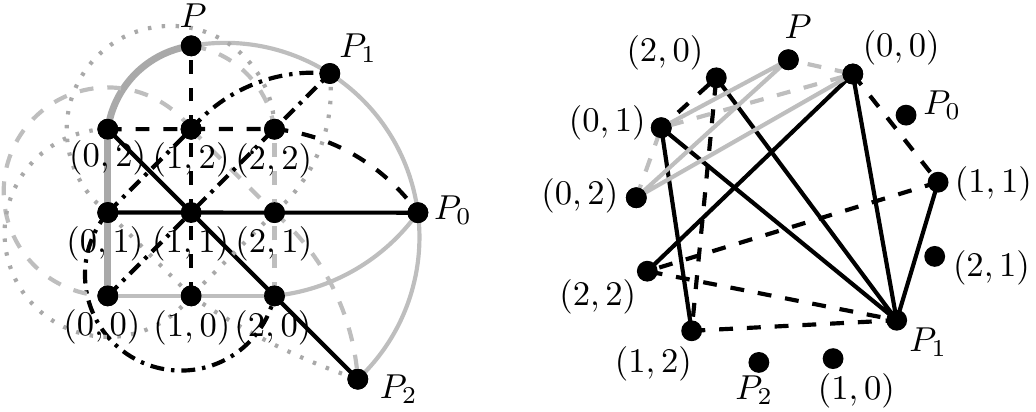}
\caption{Descomposition of $PG(2,3)$ by triplets and the complete coloring for $K_{13}$.}\label{Fig2}
\end{center}
\end{figure}

Now, we proceed to color the complete graph $G$.

To begin with, we color the complete subgraph $K_{q+1}$ associated with the line $L$ by Hamiltonian paths, therefore we use $\frac{q+1}{2}$ colors. The remaining lines are colored by pairs, according to the triplets $(p,l,m)$.

For each triplet $(p,l,m)$, we color the complete subgraph $K_{q+1}$ associated with the line $l$ by Hamiltonian paths and we copy the coloring to the complete subgraph associated with the line $m$.

Therefore, we use $\frac{q+1}{2}(\frac{n-1}{2}+1)=\frac{q+1}{4}(n+1)$ colors. The coloring is acyclic because the color classes of the line $L$ are paths, ante each color class of a triplet is the identification of two paths by a vertex.

Now, if two color classes are in the edges of a line, clearly they contain a cycle since they are the union of two Hamiltonian paths. If a color class is the the edges of $L$ and another color class is in the edges of the lines of a triplet $(p,l,m)$, the triangle formed by $l$, $m$ and $L$ induces a cycle. 

Finally, if two color classes are in the edges of the lines of the triplets $(p,l,m)$ and $(p',l',m')$, since $p\not= p'$, the triangle formed by $l$, $m$ and $l'$ induces a cycle, and then the union of any two color classes induce at least a cycle. 

The right side of Figure \ref{Fig2} shows the part of $K_{13}$ where the line $L_0$ and the two lines of the triplet $(P_1,[1,0],[1,1])$ are colored. 
\end{proof}

\section{Main Result}\label{s4}

Before proving Theorem \ref{teo4} we state the following lemma.

\begin{lemma} \label{lema3}
For any graph $G$, if $H$ is a subgraph of $G$, then $A_{\alpha}(G)\geq A_{\alpha}(H)$.
\begin{proof}
Given a coloring $\varsigma$ of $H$ which performs $A_{\alpha}(H)$, we extend this to an acyclic coloring of $G$ such that any two of color classes have a cycle in a greedy way, that is, the edges of $E(G)\setminus E(H)$ are listed in some specified order, then we assign to the edge under consideration the smallest
available color preserving the properties of the coloring. 
\end{proof}
\end{lemma}

Now we have our main result.

\begin{theorem}\label{teo4}
Let $n\geq 13$ be an integer then the achromatic arboricity of the complete graph of order $n$ is bounding by:
\[\frac{1}{4}n^{\frac{3}{2}}-\Theta(n) \leq A_{\alpha}(K_{n})\leq \frac{1}{\sqrt{2}}n^{\frac{3}{2}}-\Theta(n).\]
\end{theorem}
\begin{proof}
The upper bound is given in theorem \ref{teo1}. To prove the lower bound, we uses a strengthened version of Bertrand's Postulate, which follows from the Prime Number Theorem, see \cite{MR506522,MR0258720}: For $\epsilon>0$, there exists an $N_\epsilon$, such that for all real $x\geq N_\epsilon$ there exists a prime $q$ between $x$ and $(1+\epsilon)x$. Let $\epsilon>0$ be given, and suppose $n>(N_\epsilon+1)^2(1+\epsilon)^2$. Let $x=\sqrt{n}/(1+\epsilon) -1$, so $x\geq N_\epsilon$. We now select a prime $q$ with $x\leq q \leq (1+\epsilon)x$. Then $q^2+q+1\leq (x+1)^2(1+\epsilon)^2=n$. Since projective planes of all prime orders exist it follows from Theorem \ref{teo2} and Lemma \ref{lema3} that: \[A(K_n)\geq A(K_{q^2+q+1})\geq \frac{q+1}{4}(q^2+q+2)=\frac{n^{3/2}}{4(1+\epsilon)^3}-\frac{n}{4(1+\epsilon)^2}+\frac{\sqrt{n}}{2(1+\epsilon)}.\]
Since $\epsilon$ was arbitrarily small the result follows.
\end{proof}
 
\section{The achromatic arboricity of $K_n$ for small values of $n$.}\label{smallvalues}

In this section we study the achromatic arboricity of $K_n$ for small values of $n$. Table \ref{Tab1} shows the exact values for $2\leq n\leq 7$,  Figure \ref{Fig3} displays colorations that attain lower bounds for $2\leq n\leq 7$ and Table \ref{Tab2} shows upper and lower bounds for $8\leq n\leq 12$.

\begin{table}[!htbp]
\begin{center}
\begin{tabular}{|c|cccccc|}
\hline \hline
$n$ & 2 & 3 & 4 & 5 & 6 & 7 \\
\hline
$A_{\alpha}(K_n)$ & 1 & 2 & 3 & 4 & 6 & 7 \\
\hline
\end{tabular}
\caption{\label{Tab1}Exact values for $A_{\alpha}(K_n)$, $2\leq n\leq 7.$}
\end{center}
\end{table}

It is easy to see that the upper bounds for $n=2,3$ equals the values of Table \ref{Tab1}. 

\begin{figure}[!htbp]
\begin{center}
\includegraphics{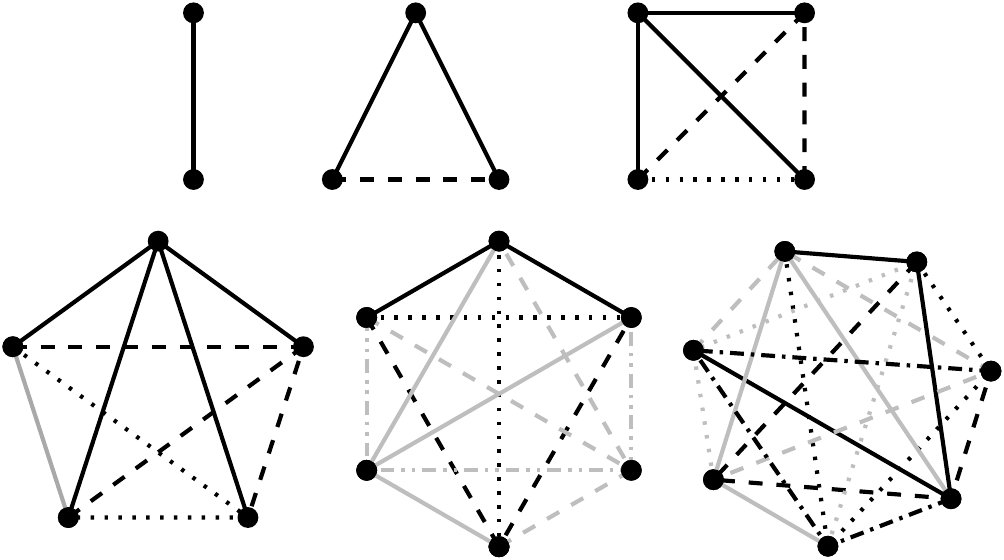}
\caption{Colorings that give lower bounds of $A_\alpha(K_n)$ for $2\leq n\leq 7$.}\label{Fig3}
\end{center}
\end{figure}

For the case $n=4$, if we suppose $A_{\alpha}(K_4)\geq 4$, then there are at least two color classes of size one, whose union does not contain a cycle, a contradiction, then  $A_{\alpha}(K_4)=3$.

For the case of $n=5$,  Lemma \ref{lema1} says that $A_{\alpha}(K_5)\leq 5$ and the smallest color class has two edges. We suppose $A_{\alpha}(K_5)= 5$ and then each color class has exactly two edges. If the edges of a color class induce a $P_3$, say $abc$,  the edge $ac$ is in the cycle with vertices $a$, $b$ and $c$, so, the others six edges incident to $P_3$ generate 3 different color classes necessarily (each on this color classes has to be connected). Hence, the edge $ac$ is in a color class which is a matching, say $ac$ and $de$. Take the color class containing the edge $bd$. On one hand, the color class is the set of two edges $bd$ and $bc$ (or $bd$ and $ab$) making a cycle with $abd$.  On the other hand, the color class is the set of two edges $bd$ and $be$ making a cycle with $ac$ and $be$, a contradiction. 
Therefore, each color class is a matching of size two.  Then the union of two color classes is a $C_4$ and the union of three color classes is a $K_4$ necessarily. The remaining edges generate a $K_{1,4}$ which does not contain matchings of size two.  In consequence,  $A_{\alpha}(K_5)= 4$. 

For the case of $n=6$,   Lemma \ref{lema1} says that $A_{\alpha}(K_6)\leq 7$ and the smallest color class has two edges. If we suppose $A_{\alpha}(K_6)= 7$, then there are six color classes of size two and one color class of size three.  If the edges of a color class of size two induce a $P_3$,  it would be incident to at most $5$ more color classes, then each color class of size two has to be a matching.  Then the union of two color classes of size two is a $C_4$ and the union of three color classes of size two is a $K_4$ necessarily, a contradiction because there are at least six color classes of size two.  Therefore,  $A_{\alpha}(K_6)= 6$. 

For the case of $n=7$, Lemma \ref{lema1} says that $A_\alpha(K_7)\leq 9$ and the smallest color class has two edges.  If we suppose $A_\alpha(K_7)= 9$,  there is at least six color classes of size two, however, a $P_3$ is incident to at most six color classes, hence, the color classes of size two are matchings but only there are at most three of them (on a $K_4$), a contradiction.  If we suppose $A_\alpha(K_7)= 8$,  there are three color classes of size two which are matching in a $K_4$ subgraph and there are five color classes of size three. On one hand, each color class of three vertices must be incident to three vertices of the $K_4$. On the other hand, there are at most four color classes of size three incident to three vertices of the $K_4$, a contradiction and then $A_\alpha(K_7)= 7$.

To end, we calculate the upper bounds given in Lemma \ref{lema1} for $A_\alpha (K_n)$ for $8 \leq n \leq 12$, and the lower bounds perform their values via greedy colorings.
\begin{table}[!htbp]
\begin{center}
\begin{tabular}{|c|ccccc|}
\hline \hline
$n$ & 8 & 9 & 10 & 11 & 12 \\
\hline
Upper bound & 11 & 13 & 15 & 18 & 22 \\
\hline
Lower bound & 8 & 9 & 10 & 11 & 12 \\
\hline
\end{tabular}
\caption{\label{Tab2}Bounds for $A_{\alpha}(K_n)$, $8\leq n\leq 12.$}
\end{center}
\end{table}

\end{document}